\documentclass[12pt]{amsart}

\usepackage{amsthm, amsmath}
\usepackage{amssymb}
\usepackage[all]{xy}
\usepackage{mathrsfs}
\usepackage[latin1]{inputenc}
\usepackage{graphicx}
\usepackage{xspace}

\numberwithin{equation}{subsection}

\DeclareMathOperator{\Gal}{Gal}

\DeclareMathOperator{\supp}{supp}

\DeclareMathOperator{\id}{id}
\DeclareMathOperator{\Adm}{Adm}
\DeclareMathOperator{\EO}{EO}

\DeclareMathOperator{\cox}{cox}

\DeclareMathOperator{\pos}{pos}
\newcommand{\CSC}{(\textbf{CSC})\xspace}
\newcommand{\CC}{(\textbf{CC})\xspace}

\theoremstyle{plain}

\newtheorem{theorem}{Theorem}[subsection]

\newtheorem*{thrm1}{Theorem A}

\newtheorem*{thrm2}{Theorem B}

\newtheorem*{thrm3}{Theorem C}

\newtheorem{thm}[theorem]{Theorem}
\newtheorem{cor}[theorem]{Corollary}

\newtheorem{prop}[theorem]{Proposition}

\theoremstyle{definition}

\newtheorem{example}[theorem]{Example}

\def\ge{\geqslant}
\def\le{\leqslant}
\def\a{\alpha}

\def\G{\Gamma}

\def\e{\epsilon}

\def\o{\omega}

\def\s{\sigma}
\def\t{\tau}
\def\th{\theta}
\def\k{\kappa}
\def\l{\lambda}

\def\i{^{-1}}
\def\tSS{\tilde{\mathbb S}}
\def\SS{\mathbb S}
\def\JJ{\mathbb J}
\def\ZZ{\mathbb Z}
\def\NN{\mathbb N}
\def\QQ{\mathbb Q}

\def\ca{\mathcal A}

\def\ci{\mathcal I}

\def\cn{\mathcal N}
\def\co{\mathcal O}
\def\cp{\mathcal P}

\def\ct{\mathcal T}

\def\tW{\tilde W}

\usepackage{amssymb} % symboler fra AMS
\usepackage{latexsym} % LaTeX symboler
\usepackage{amsfonts} % fonte fra AMS
\usepackage{amsmath} % matematik fonte fra AMS
\usepackage{eucal} % 'krlle'-fonte (skriv \mathcal)
\usepackage{bm} % fed skrift i matematik
\usepackage{bbm} % skrift med dobbelt streg til f.eks. talomr�er
\usepackage{graphicx} % mulighed for at inkludere grafik
\usepackage[english]{varioref} % smarte referencer - se nederst
\usepackage[nice]{nicefrac} % p�ere brker
\usepackage[all]{xy}

\newcommand{\kk}{\Bbbk}

\newcommand{\FF}{\mathbbm{F}}

\def\<{\langle}
\def\>{\rangle}

\begin{document}

\title[Basic loci of Coxeter type]{Basic loci in Shimura varieties of Coxeter type}

\author[U. G\"{o}rtz]{Ulrich G\"{o}rtz}
\address{Ulrich G\"{o}rtz\\Institut f\"ur Experimentelle Mathematik\\Universit\"at Duisburg-Essen\\Ellernstr.~29\\45326 Essen\\Germany}
\email{ulrich.goertz@uni-due.de}
\thanks{U. G\"{o}rtz was partially supported by the Sonderforschungsbereich TR 45 ``Periods, Moduli spaces and Arithmetic of Algebraic Varieties'' of the Deutsche Forschungsgemeinschaft.}

\author[X. He]{Xuhua He}
\address{Xuhua He, Department of Mathematics and Institute for advanced study, The Hong Kong University of Science and Technology, Clear Water Bay, Kowloon, Hong Kong}
\address{Xuhua He, Department of Mathematics, University of Maryland, College Park, MD 20742, USA}
\thanks{X. He was partially supported by HKRGC grant 602011.}
\email{xuhuahe@gmail.com}

\begin{abstract}
This paper is a contribution to the general problem of giving an explicit description of the basic locus in the reduction modulo $p$ of Shimura varieties. Motivated by \cite{Vollaard-Wedhorn} and \cite{Rapoport-Terstiege-Wilson}, we classify the cases where the basic locus is (in a natural way) the union of classical Deligne-Lusztig sets associated to Coxeter elements.  We show that if this is satisfied, then the Newton strata and Ekedahl-Oort strata have many nice properties.  
\end{abstract}

\maketitle

\section*{Introduction}

Understanding arithmetic properties of Shimura varieties has been a cornerstone in many developments in arithmetic geometry and number theory in the last decades. To a large extent, these arithmetic properties are encoded in the geometric properties of the special fiber of a suitable integral model, and studying these reductions of Shimura varieties has been fruitful in many cases.

A Shimura variety of PEL type can be described as a moduli space of abelian varieties with additional structure (polarization, endomorphisms, level structure). To each such abelian variety we can attach its $p$-divisible group; it inherits corresponding additional structure. The special fiber at $p$ naturally decomposes into finitely many ``Newton strata'', which are given by the isogeny class of these $p$-divisible groups (with additional structure).

There is a unique closed Newton stratum. This is the so-called \emph{basic locus}. For two reasons the basic locus plays a particular role in the study of the geometry of the special fiber. First, it is the only Newton stratum where there is reasonable hope for a complete, explicit description as a variety. Second, it turns out that a good understanding of basic loci can often be used to prove results about general Newton strata, and hence about the whole special fiber, by an induction process.

Explicit descriptions of the basic locus have been of great importance in the work of Kudla, Rapoport, Howard, Terstiege and others on the intersection numbers of special cycles in the special fibers of Shimura varieties (and their relationship to Fourier coefficients of modular forms as predicted by the ``Kudla program''). Kaiser \cite{Kaiser} used the description of the basic locus in the module space of principally polarized abelian surfaces in his proof of the twisted fundamental lemma for $GSp_4$. A good description of the basic locus is also useful to prove instances of the ``arithmetic fundamental lemma'', cf.~the paper by Rapoport, Terstiege and Zhang \cite{rapoport-terstiege-zhang}.

While the basic locus is the simplest Newton stratum, it still cannot be described explicitly in general. For example, the basic locus in the moduli space of principally polarized abelian varieties of dimension $g$ is just the supersingular locus, i.e., the closed subvariety of supersingular abelian varieties. %This variety has been studied extensively % references Oort, Li-Oort If $g=1$, then we are looking at a modular curve; in this case the supersingular locus consists just of finitely many points. Already for $g=2$, where the supersingular locus has been described in detail by Kaiser \cite{Kaiser}, the result is much more complicated. Richartz \cite{Richartz} has described the supersingular locus for $g=3$. Probably for $g\ge 4$ the result is too complicated to be of any use. One of the results of our paper is a clear understanding why those cases are so complicated (together with a ``recipe'' for nevertheless guessing their structure, if one is willing to go through a lot of complicated combinatorics).

On the other hand, besides several small rank cases (cf.~Section~\ref{sec:5.3}), there are a number of families of Shimura varieties of PEL-type where the basic locus allows for a simple and very explicit description. The case studied first is probably the ``Drinfeld case'' where the underlying algebraic group is a division algebra and the basic locus can be described in terms of Deligne's formal model of Drinfeld's half space. Other typical cases are those attached to unitary groups $GU(1, n-1)$. See the papers by Vollaard and Wedhorn \cite{Vollaard-Wedhorn}, and by Rapoport, Terstiege, and Wilson \cite{Rapoport-Terstiege-Wilson}. Roughly speaking, in all those cases the following picture emerges: The basic locus is a union of Ekedahl-Oort strata and admits a stratification by (variants of) classical Deligne-Lusztig varieties. The index set of the stratification and the closure relations between strata can be described in terms of the Bruhat-Tits building of a certain inner form of the underlying group.

The uniformization theorem by Rapoport and Zink \cite{RZ} allows to describe the basic locus in terms of a moduli space of $p$-divisible groups with additional structure, a so-called Rapoport-Zink space. Roughly speaking, the set of points of this Rapoport-Zink space can be described, using Dieudonn\'e theory, as a space of lattices inside a fixed $L$-vector space. Here $L$ is the completion of the maximal unramified extension of $\mathbb Q_p$ (or more generally of a finite extension $F/\mathbb Q_p$). The lattices have to satisfy conditions which ensure that they arise as the Dieudonn\'e module of a $p$-divisible group with additional structure as specified by the moduli problem. In other terminology, these lattices form an \emph{affine Deligne-Lusztig variety} inside the space of all lattices.

In this paper, we are mainly interested in the ``Coxeter case'', i.e., the cases where the basic locus is a union of Ekedahl-Oort strata and each Ekedahl-Oort stratum is (in a natural way) the union of classical Deligne-Lusztig sets associated to Coxeter elements. 

Let $F$ be a finite extension of $\mathbb Q_p$ (the mixed characteristic case) or let $F=\mathbb F_q((\epsilon))$ (the function field case). Fix a datum $(G, \mu)$ of a connected quasi-simple semisimple algebraic group $G$ over $F$ which splits over a tamely ramified extension and a minuscule cocharacter $\mu$ (see Section~\ref{sec:notation} and Section~\ref{setup}). Let $P$ be a standard rational maximal parahoric subgroup of $G(L)$, where $L$ is the completion of the maximal unramified extension of $F$. Let $\s$ denote the Frobenius of the extension $L/F$. 

For each $b\in G(L)$ we consider the following union of affine Deligne-Lusztig varieties:
\[
X(\mu, b)_P = \{ g\in G(L)/P;\ g\i b\s(g)\in \bigcup_{w\in\Adm(\mu)} PwP \}
\]
(see~\ref{sec:7.1}, compare also \cite{rapoport:guide} Section 5). Denote by $B(G,\mu)$ the set of $\s$-conjugacy classes for which this set is non-empty. There is a unique basic $\s$-conjugacy class in $B(G, \mu)$. If $b$ lies inside this basic $\s$-conjugacy class, then we call $X(b,\mu)_P$ the \emph{basic locus} attached to the above data. (For a different choice of $b$ we get, up to isomorphism, the same result. In fact, in a large part of the paper we work with all basic $b$ simultaneously, i.e.~with ``fiber bundles'' over the basic $\s$-conjugacy class, and only later to restrict to single fibers of this bundle.)

In case the data $(G,\mu,b)$ corresponds to a Rapoport-Zink space $\mathscr M$ (in particular, $\mathop{\rm char} F=0$), there typically is a bijection $\mathscr M(\overline{\mathbb F}_p) \cong X(\mu, b)_P$, given by Dieudonn\'e theory; cf.~\cite{rapoport:guide}.

In Section~\ref{setup} we define a notion of \emph{Ekedahl-Oort elements} which gives rise to a subset $\EO^J(\mu)\subset \tW$ of the extended affine Weyl group $\tW$ (here $J$ denotes the type of the parahoric subgroup $P$). For each $b$ and each $w\in \EO^J(\mu)$, we obtain the EO stratum $X_{J,w}(b)$, cf.~\eqref{EOstrat}, attached to $w$ inside $X(\mu, b)_P$. We say that the basic locus is a union of EO strata, if for all $w\in \EO^J(\mu)$ such that $X_{J,w}(b)\ne\emptyset$ for $b$ basic, we have $X_{J,w}(b')=\emptyset$ for all non-basic $b'$. For the formal definition of ``Coxeter type'' see Theorem~\ref{classification} and condition \CC in Section~\ref{6}.

The main results of this paper are summarized below. 

\begin{thrm1}
The data $(G, \mu, P)$ of Coxeter type are listed in Theorem \ref{classification}. 
\end{thrm1}

\begin{thrm2}
If $(G, \mu, P)$ is of Coxeter type, then 

\begin{enumerate}
\item 
The basic locus \[X(\mu, b)_P=\sqcup_{\Lambda} \cn_\Lambda^o,\] where $\Lambda$ runs over faces of certain types of the rational Bruhat-Tits building $\mathscr B$ of the $\s$-centralizer $\JJ_b$ of $b$, and $\cn_\Lambda^o$ is a classical Deligne-Lusztig set associated to a Coxeter element. 
\item
For all non-basic $b'$, the $\s$-centralizer $\JJ_{b'}$ acts transitively on $X(\mu, b')_P$. In particular, whenever there is a notion of dimension (in the RZ cases and in the function field case), this implies that $\dim X(\mu, b')_P = 0$.
\end{enumerate}
\end{thrm2}

In most cases, $\Lambda$ runs over all vertices  in $\mathscr B$. In other cases, edges of $\mathscr B$ appear. For more details, see Theorem \ref{classification}, Theorem \ref{main} and $\S$\ref{7}. 

\begin{thrm3}
In the function field case, if $(G, \mu, P)$ is of Coxeter type, then 
the closure relation among the strata $\cn^o_\Lambda$ can be described explicitly in terms of the simplicial structure of $\mathscr B$.
\end{thrm3}

The key ingredients in this paper are 
\begin{itemize}
\item The description of affine Deligne-Lusztig varieties in the affine flag variety \cite{He-GeometryOfADLV}.

\item The fine Deligne-Lusztig varieties \cite{Lu-par}.
\end{itemize}

In mixed characteristic, in general there is no known scheme structure on affine Deligne-Lusztig varieties. This is the technical difficulty preventing us from extending Theorem C to the mixed characteristic case. However, experience shows that, at least for the basic locus, the descriptions in the mixed characteristic case and in the equal characteristic case are mostly equal (see Sections~\ref{sec:5.3}, \ref{sec:examples}). Therefore we expect that in those cases which have not been treated in the context of Shimura varieties, our results can serve as guide lines for the precise result to be expected.

We focus on ``Coxeter type'' in this paper. However, our methods should extend to some other cases where the basic locus is still a union of EO strata but the EO strata there are not of Coxeter type in general. We include one example at the end of this paper. 

The paper is organized as follows. In section 1 we fix notation and give a group-theoretic definition of the basic locus. In section 2, we recollect properties of affine Deligne-Lusztig varieties in the affine flag variety. In section 3, we give a stratification, which includes Kottwitz-Rapoport stratification and Ekedahl-Oort stratification as special cases. In section 4, we study the fine affine Deligne-Lusztig varieties in the affine Grassmannian. Theorem A, Theorem B and the first two parts of Theorem C are stated in section 5 and proofs are given in section 6. Theorem C is proved in section 7. We also study the singularities of $\overline{\cn^o_\lambda}$ in section 7.

\emph{Acknowledgments.} We thank George Pappas, Michael Rapoport, Eva Viehmann and Xinwen Zhu for useful discussions, answering questions and pointing out mistakes in an earlier version. % FIXME ... others we should thank?

After the paper was finished, we received an email from Xinwen Zhu about his joint work in preparation with Liang Xiao \cite{XZ}.They study the basic affine Deligne-Lusztig varieties in the affine Grassmannian for some unramified quasi-split (but not split) groups. This gives a different approach to the basic loci of types $({}^2 A'_n, \o^\vee_1, \SS)$ for $n$ even, $({}^2 D_n, \o^\vee_1, \SS)$ and $({}^2 A_3, \o^\vee_2, \SS)$ in our list in Theorem \ref{classification}. They also give a description for some basic loci of non-Coxeter type. 

\section{Preliminaries}
\label{SEC:preliminaries}

\subsection{Notation}\label{sec:notation}

Let $\mathbb F_q$ be the finite field with $q$ elements. Let $\kk$ be an algebraic closure of $\mathbb F_q$. Let $F= \mathbb F_q( (\e))$ or a finite field extension of $\QQ_p$ with residue class field $\mathbb F_q$ and uniformizer $\e$, and let $L$ be the completion of the maximal unramified extension of $F$. 

Let $G$ be a connected semisimple group over $F$ which splits over a tamely ramified extension of $F$.  Let $\s$ be the Frobenius automorphism of $L/F$. We also denote the induced automorphism on $G(L)$ by $\s$.

Let $S \subset G$ be a maximal $L$-split torus defined over $F$ and let $T$ be its centralizer. Since $G$ is quasi-split over $L$, $T$ is a maximal torus of $G$. The {\it Iwahori-Weyl group} associated to $S$ is 
\[
\tW = N_S(L)/T(L)_1.
\]
Here $N_S$ denotes the normalizer of $S$ in $G$, and $T(L)_1$ denotes the unique parahoric subgroup of $T(L)$. For $w \in \tW$, we choose a representative in $N_S(L)$ and also write it as $w$. 

\subsection{Weyl groups} We denote by $\mathcal A$ the apartment of $G_L$ corresponding to $S$. We fix a $\sigma$-invariant alcove $\mathfrak a$ in $\ca$, and denote by $I\subseteq G(L)$ the Iwahori subgroup corresponding to $\mathfrak a$ over $L$.

The Iwahori-Weyl group $\tW$ is an extension of the {\it relative Weyl group} $W_0=N_S(L)/T(L)$ by $X_*(T)_\G$, where $\G=\Gal(\bar L/L)$ is the absolute Galois group of $L$. If we choose a special vertex of $\mathfrak a$, we may represent the Iwahori-Weyl group as a semidirect product $$\tW=X_*(T)_\G \rtimes W_0=\{t^\l w; \l \in X_*(T)_\G, w \in W_0\}.$$ 

We denote by $\tSS$ the set of simple reflections of $\tW$ and denote by $\SS \subset \tSS$ the set of simple reflections of $W_0$. Both $\tW$ and $\tSS$ are equipped with an action of $\s$. 

For any subset $J$ of $\tSS$, we denote by $W_J$ the subgroup of $\tW$ generated by the simple reflections in $J$ and by ${}^J \tW$ (resp. $\tW^J$) the set of minimal length elements for the cosets $W_J \backslash \tW$ (resp. $\tW/W_J$). We simply write ${}^J \tW^K$ for ${}^J \tW \cap \tW^K$. 

The subgroup $W_{\tSS}$ of $\tW$ is the affine Weyl group and we usually denote it by $W_a$. Then $$\tW=W_a \rtimes \Omega,$$ where $\Omega$ is the normalizer of the base alcove $\mathfrak a$. We may identify $\Omega$ with $\pi_1(G)_\G$.

\subsection{$\s$-conjugacy classes} We say that $b, b' \in G(L)$ are $\s$-conjugate if $b'=g \i b \s(g)$ for some $g \in G(L)$. We denote by $B(G)$ the set of $\s$-conjugacy classes of $G(L)$. The classification of the $\s$-conjugacy classes is obtained by Kottwitz in \cite{kottwitz-isoI} and \cite{kottwitz-isoII}. The description is as follows. 

An element $b \in G(L)$ determines a homomorphism $\mathbb D \to G_L$, where $\mathbb D$ is the pro-algebraic torus whose character group is $\QQ$. This homomorphism determines an element $\bar \nu_b$ in the closed dominant chamber $X_*(T)_\QQ^+$. The element $\bar \nu_b$ is called the {\it Newton point} of $b$ and the map $b \mapsto \bar \nu_b$ is called the {\it Newton map}. Let $\k_G: B(G) \to \pi_1(G)_{\G_F}$ be the Kottwitz map \cite[\S 7]{kottwitz-isoII}, where $\G_F$ is the absolute Galois group of $F$. By \cite[\S 4.13]{kottwitz-isoII}, the map $$B(G) \to X_*(T)_\QQ^+ \times \pi_1(G)_{\G_F}, \qquad b \mapsto (\bar \nu_b, \k_G(b))$$ is injective. The set $B(G)$ is equipped with a partial order, see~\cite{RR} Section 2.3.

\subsection{Straight conjugacy classes} Following \cite{He-GeometryOfADLV}, we relate $B(G)$ to the Iwahori-Weyl group $\tW$. 

For any $w \in \tW$, we consider the element $w \s \in \tW \rtimes \<\s\>$. There exists $n \in \NN$ such that $(w \s)^n=t^\l$ for some $\l \in X_*(T)_\G$. Let $\bar \nu_{w}$ be the unique dominant element in the $W_0$-orbit of $\l/n$. It is known that $\bar \nu_{w}$ is independent of the choice of $n$ and is $\G$-invariant. Moreover, $\bar \nu_w$ is the Newton point of $w$ when regarding $w$ as an element in $G(L)$. 

We say that an element $w$ is {\it $\s$-straight} if $\ell((w\s)^n)=n \ell(w)$. This is equivalent to saying that $\ell(w)=\<\bar \nu_{w}, 2 \rho\>$, where $\rho$ is the half sum of all positive roots in the root system of the affine Weyl group $W_a$. A $\s$-conjugacy class of $\tW$ is called {\it straight} if it contains a $\s$-straight element. 

The map $N_S(L) \to G(L)$ induces a map $\tW \to B(G)$. Here $\k_G(w)$ is the image of $w$ under the projection $\tW \to \Omega \cong \pi_1(G)_\G \to \pi_1(G)_{\G_F}$. By \cite[\S 3]{He-GeometryOfADLV}, the map induces a bijection from the set of straight $\s$-conjugacy classes of $\tW$ to $B(G)$. 

A $\s$-conjugacy class $[b]$ of $G(L)$ is called {\it basic} if $\bar \nu_b$ factors through the center of $G$. Again by \cite[\S 3]{He-GeometryOfADLV}, a $\s$-conjugacy class of $G(L)$ is basic if and only if it contains some element of $\Omega$.

\subsection{The variety $Z$} Let $\mu \in X_*(T)$ be a minuscule coweight. We denote by $\l$ its image in the coinvariants $X_*(T)_\G$. The {\it admissible subset} of $\tW$ associated to $\mu$ is defined as $$\Adm(\mu)=\{w \in \tW; w \le t^{x(\l)} \text{ for some } x \in W_0\}.$$ Note that $\l$ is not minuscule in $\tW$ in general (see also $\S$\ref{7.2}). We also denote by $\t$ the image of $t^\l$ under the projection map $\tW=W_a \rtimes \Omega \to \Omega$. 

Let $J \subset \tSS$. Let $P_J \supset I$ be the standard parahoric subgroup corresponding to $J$. Set $\Adm^J(\mu)=W_J \Adm(\mu) W_{\s(J)}$ and $$Y_J=\cup_{w \in \Adm(\mu)} P_J w P_{\s(J)}=\cup_{w \in \Adm^J(\mu)} I w I.$$ 

Define the action of $P_J$ on $G(L) \times Y_J$ by $p \cdot (g, y)=(g p \i, p y \s(p) \i)$ and denote by $Z_J$ its quotient. Then the map $(g, y) \mapsto (g y \s(g) \i, g P_J)$ gives an isomorphism $$Z_J \cong \{(b, g P_J) \in G(L) \times G(L)/P_J; g \i b \s(g) \in Y\}.$$

The image of the projection map $Z_J \to G(L)$ is a union of $\s$-conjugacy classes of $G(L)$ and we denote it by $B(G, \mu)_J$. In fact, $B(G, \mu)_J$ is independent of the subset $J$ we choose \cite{He-000}. However, we don't need this fact in this paper.

The basic $\s$-conjugacy class in $B(G, \mu)_J$ contains the element $\t$ and we denote this $\s$-conjugacy class by $\co_0$. 
We have the Newton stratification $$Z_J=\sqcup_{\co \in B(G, \mu)_J} Z_{J, \co},$$ where $Z_{J, \co}=\{(b, g P_J) \in Z_J; b \in \co\}$. The stratum $Z_{J, \co_0}$ is called the {\it basic locus in $Z_J$}. 

\section{Affine Deligne-Lusztig varieties}

\subsection{Affine Deligne-Lusztig varieties}\label{2} We first look at the case where $J=\emptyset$. Then $Y_\emptyset=\sqcup_{w \in \Adm(\mu)} I w I$ and we have the Kottwitz-Rapoport stratification $$Z_\emptyset=\sqcup_{w \in \Adm(\mu)} Z_{\emptyset, w},$$ where $Z_{\emptyset, w}=G(L) \times^I I w I$ for any $w \in \Adm(\mu)$. 

Given $w \in \Adm(\mu)$ and $\co \in B(G, \mu)_\emptyset$, the intersection $Z_{\emptyset, w} \cap Z_{\emptyset, \co}$ is a fiber bundle over $\co$ and the fiber over $b \in \co$ is the affine Deligne-Lusztig variety (in the affine flag variety) $$X_w(b)=\{g I \in G(L)/I; g \i b \s(g) \in I w I\} \subset G(L)/I.$$ Moreover, the $\s$-centralizer $\JJ_b=\{g \in G(L); g \i b \s(g)=b\}$ acts on $X_w(b)$. 

In the rest of this section, we recollect some results on affine Deligne-Lusztig varieties. The following result on $\s$-straight elements is proved in \cite[Proposition 4.5 \& Theorem 4.8]{He-GeometryOfADLV}. 

\begin{thm}\label{straightADLV}
If $w$ is $\s$-straight, then $X_{w}(b) \neq \emptyset$ if and only if $b$ is $\s$-conjugate to $w$. In this case, $$X_{w}(b) \cong X_{w}(w) \cong \JJ_{w}/(\JJ_{w} \cap I)$$ (which means that, in the function field case, $X_w(b)$ is $0$-dimensional).
\end{thm}

\subsection{Support} 
For $w \in W_a$, we denote by $\supp(w)$ the support of $w$, i.e., the set of $i \in \tSS$ such that $s_i$ appears in some (or equivalently, any) reduced expression of $w$. We set $$\supp_\s(w \t)=\bigcup_{n \in \ZZ} (\t \s)^n(\supp(w)).$$ 
Then $\supp_\s(w \t)$ is the minimal $\t \s$-stable subset $J$ of $\tSS$ such that $w \t \s\in W_J \rtimes \<\t \s\>\}$. 

If $\ell(w)=\sharp(\supp_\s(w \t)/\<\t \s\>)$, i.e., $w$ is a product of simple reflections in $W_a$ and the simple reflections from each orbit of $\t \s$ appears at most once, then we say that $w \t$ is a {\it $\s$-Coxeter element}. 

%Now we consider another simple situation: $\supp_\s(\tilde{w}) \ne \tSS$. 

%\remind{What is a good terminology to express that $\supp_\s(\tilde{w})\ne\tSS$?}

\begin{prop}\label{decomp-flag}
Let $w\in W_a \t$ such that $W_{\supp_\s(w)}$ is finite. Then 
\[
X_w(\t) = \coprod_{i\in \JJ_\t/(\JJ_\t \cap P_{\supp_\s(w)})} i Y(w),
\]
where
\[
Y(w) = \{ gI\in P_{\supp_\s(w)}/I;\ g\i \tau \s(g) \in IwI \}
\]
is a classical Deligne-Lusztig variety in the finite-dimensional flag variety $P_{\supp_\s(w)}/I$.
\end{prop}

The proposition follows from the proof of \cite[Theorem 4.7]{He-GeometryOfADLV}. For the sake of completeness and because the reasoning simplifies in our setting, we reproduce the relevant part of the proof in loc.~cit.

\begin{proof}
Let $g \in G(L)$ with $g\i \t\s(g)\in IwI$. By \cite[Lemma 3.2 \& Proposition 4.5]{He-GeometryOfADLV}, there exists $p\in P_{\supp_\s(w)}$, such that $(gp)\i \t\s(gp)=\t$. Hence $g\in \JJ_\t P_{\supp_\s(w)}$, and 
\[
X_w(\t)=\{g I \in \JJ_\t P_{\supp_\s(w)}/I; g \i \t \s(g) \in I w I\} \subset \JJ_\t P_{\supp_\s(w)}/I.
\]

Note that $\JJ_\t P_{\supp_\s(w)}/I=\sqcup_{i \in \JJ_\t/(\JJ_\t \cap P_{\supp_\s(w)})} i P_{\supp_\s(w)}/I$ and for any $g \in P_{\supp_\s(w)}$ and $i \in \JJ_\t$, $(i g) \i \t \s(i g) \in I w I$ if and only if $g \i \t \s(g) \in I w I$. Hence 
\[
X_w(\t) = \coprod_{i\in \JJ_\t/(\JJ_\t \cap P_{\supp_\s(w)})} Y(w).
\]
\end{proof}

\section{$P_J$-stable pieces}

%We discussed the stratification $Z_\emptyset=\sqcup_{w \in \Adm(\mu)} Z_{\emptyset, w}$ in the last section. Now we discuss $Z_J$ for arbitrary $J$. We first recall some properties about the subset of $G(L)$ of the form $P_J \cdot_\s I w I$, where $\cdot_\s$ denotes the $\s$-conjugation. The case where $G(L)$ is split and $J=\SS$ is considered in \cite{He-aff}. The general case can be proved in a similar way and we omit the details. 

\subsection{Partial conjugation action} Let $J \subset \tSS$. The partial conjugation action of $W_J$ on $\tW$ defined by $x \cdot_\s y=x y \s(x) \i$ for $x \in W_J$ and $y \in \tW$. 

Given $w, w' \in \tW$ and $j \in J$, we write $w \xrightarrow{s_j}_\s w'$ if $w'=s_j w s_{\s(j)}$ and $l(w') \le l(w)$. If $w=w_0, w_1, \cdots, w_n=w'$ is a sequence of elements in $\tW$ such that for all $k$, we have $w_{k-1} \xrightarrow{s_j}_\s w_k$ for some $j \in J$, then we write $w \rightarrow_{J, \s} w'$. We write $w \approx_{J, \s} w'$ if $w \to_{J, \s} w'$ and $w' \to_{J, \s} w$.

The following property is proved in \cite[Proposition 3.4]{HeMin}.

\begin{prop}\label{min}
For any $w \in \tW$, there exists a minimal length element $w' \in W_J \cdot_\s w$ such that $w \to_{J, \s} w'$. Moreover, we may take $w'$ to be of the form $v w_1$ with $w_1 \in {}^J \tW$ and $v \in W_{I(J, w_1, \s)}$. Here $I(J, w_1, \s)=\max\{K \subset J; \text{Ad}(w_1) \s(K)=K\}$. 

%(2) Let $w \in {}^J \tW$ and $\co=W_J \cdot_\s w$. Then $w' \approx_{J, \s} w$ for any minimal length element $w' \in \co$. 
\end{prop}

\subsection{The subset $P_J \cdot_\s I w I$} \label{py} By \cite[section 2]{He-aff}, 

(1) If $w \approx_{J, \s} w'$, then $P_J \cdot_\s I w I=P_J \cdot_\s I w' I$. 

(2) If $w \xrightarrow{s_i}_\s w'$ with $\ell(w')<\ell(w)$, then $P_J \cdot_\s I w I=P_J \cdot_\s I w' I \cup P_J \cdot_\s I s_i w I$. 

(3) If $w \in {}^J \tW$ and $x \in W_{I(J, w, \s)}$, then $P_J \cdot_\s I x w I=P_J \cdot_\s I w I$. 

A subset of $G(L)$ of the form $P_J \cdot_\s I w I$ for some $w \in {}^J \tW$ is called a {\it $P_J$-stable piece}. It is analogous to the $G$-stable pieces introduced by Lusztig in \cite{Lu-Par1}. It is showed in \cite[1.4]{Lu} and \cite[Proposition 2.5 \& 2.6]{He-aff} that $$G(L)=\sqcup_{w \in {}^J \tW} P_J \cdot_\s I w I.$$

The following result is essentially contained in the proof of \cite[Proposition 2.5]{He-aff}. 

\begin{thm}
Let $w \in \tW$. Then $$P_J w P_{\s(J)}=\sqcup_{x \in W_J w W_{\s(J)} \cap {}^J \tW} P_J \cdot_\s I x I.$$ 
\end{thm}

\begin{proof}
It is obvious that $\sqcup_{x \in W_J w W_{\s(J)} \cap {}^J \tW} P_J \cdot_\s I x I \subset P_J w P_{\s(J)}$. Now we prove that $P_J w P_{\s(J)} \subset \sqcup_{x \in W_J w W_{\s(J)} \cap {}^J \tW} P_J \cdot_\s I x I$. 

Notice that $P_J w P_{\s(J)}=\sqcup_{w' \in W_J w W_{\s(J)}} I w' I$. We argue by induction that $I w' I \subset \sqcup_{x \in W_J w W_{\s(J)} \cap {}^J \tW} P_J \cdot_\s I x I$ for any $w' \in W_J w W_{\s(J)}$. 

If $w'$ is a minimal length element in $W_J \cdot_\s w'$, then $w' \approx_{J, \s} v x$ for some $x \in {}^J \tW$ and $v \in W_{I(J, x, \s)}$. In this case, $x \in W_J w' W_{\s(J)}=W_J w W_{\s(J)}$. The statement follows from $\S$\ref{py} (1) \& (3). 

If $w'$ is not a minimal length element in $W_J \cdot_\s w'$, then there exists $w'' \approx_{J, \s} w'$ and $i \in J$ such that $\ell(s_i w'' s_{\s(i)})<\ell(w')$. The statement follows from the induction hypothesis and $\S$\ref{py} (1) \& (2). 
\end{proof}

\subsection{A partial order on ${}^J \tW$}\label{4.3} We introduce $\le_{J, \s}$ as follows. For $w \in {}^J \tW$ and $w' \in \tW$, we write $w \le_{J, \s} w'$ if there exists $x \in W_J$ such that $x w \s(x) \i \le w'$. By \cite[4.7]{HeMin}, $\le_{J, \s}$ gives a partial order on ${}^J \tW$. For $w, w' \in {}^J \tW$, 

(1) $w \le w'$ implies that $w \le_{J, \s} w'$;

(2) $w \le_{J, \s} w'$ implies that $\ell(w) \le \ell(w')$. 

It is proved in \cite[Proposition 2.6]{He-aff} that if $F=\FF_q((\e))$, then $$\overline{P_J \cdot_\s I w I}=\sqcup_{x \in {}^J \tW, x \le_{J, \s} w} P_J \cdot_\s I x I.$$ 

\subsection{A stratification of $Z_J$} Since $Y_J=\sqcup_{w \in \Adm^J(\mu) \cap {}^J \tW} P_J \cdot_\s I w I$, we have the stratification $$Z_J=\sqcup_{w \in \Adm^J(\mu) \cap {}^J \tW} Z_{J, w},$$ where $Z_{J, w}=G(L) \times^{P_J} (P_J \cdot_\s I w I)$. This includes as special cases the Kottwitz-Rapoport stratification discussed in $\S$\ref{2} and the Ekedahl-Oort stratification we will discuss in $\S$\ref{setup}. See also~\cite{He-Wedhorn}.

Given $w \in \Adm^J(\mu) \cap {}^J \tW$ and $\co \in B(G, \mu)_J$, the intersection $Z_{J, w} \cap Z_{J, \co}$ is a fiber bundle over $\co$ and the fiber over $b \in \co$ is
\begin{equation}\label{EOstrat}
X_{J,w}(b):= \{g P_J; g \i b \s(g) \in P_J \cdot_\s I w I\} \subset G(L)/P_J.
\end{equation}
It is the image of $X_{w}(b)$ under the projection map $\pi_J: G(L)/I \to G(L)/P_J$. We call $X_{J, w}(b)$ a {\it fine affine Deligne-Lusztig variety} in $G(L)/P_J$. 

\section{Fine affine Deligne-Lusztig varieties}

\subsection{(Coarse) affine Deligne-Lusztig varieties} For any $J \subset \tSS$, we have another stratification $$G(L)/P_J=\sqcup_{w \in {}^J \tW^{\s(J)}} \{g P_J; g \i b \s(g) \in P_J w P_{\s(J)}\}.$$ Each subset $\{g P_J; g \i b \s(g) \in P_J w P_{\s(J)}\}$ is a union of fine affine Deligne-Lusztig varieties. We call it a {\it (coarse) affine Deligne-Lusztig variety} in $G(L)/P_J$. Similar to the proof of Proposition \ref{decomp-flag}, we have 

\begin{prop}\label{decomp-flag2}
Let $J \subset \tSS$ and $w \in {}^J \tW^{\s(J)} \cap W_a \t$ such that $\text{Ad}(w) \s(J)=J$. If $W_{\supp_\s(w)\cup J}$ is finite, then $$\{g P_J; g \i \t \s(g) \in P_J w P_{\s(J)}\}=\sqcup_{i \in \JJ_\t/(\JJ_\t \cap P_{\supp_\s(w) \cup J})} i Y_J(w),$$ where $Y_J(w)=\{g P_J \in P_{\supp_\s(w) \cup J}/P_J; g \i \t \s(g) \in P_J w P_{\s(J)}\}$ is a classical Deligne-Lusztig variety in the partial flag variety $P_{\supp_\s(w) \cup J}/P_J$. 
\end{prop}

The main result we prove in this section (see Sections~\ref{sec:4.5}, \ref{sec:4.6}) is the following theorem which relates fine affine Deligne-Lusztig varieties with (coarse) affine Deligne-Lusztig varieties. 

\begin{thm}\label{fineADLV}
For any $J \subset \tSS$ and $w \in {}^J \tW$, 

(1) $X_{J, w}(b) \cong \{g P_{I(J, w, \s)}; g \i b \s(g) \in P_{I(J, w, \s)} w P_{\s(I(J, w, \s))}\}$. 

(2) If $F=\FF_q((\e))$, then $\dim X_{J, w}(b)=\dim X_w(b)$. 
\end{thm}

\begin{cor}\label{4.1.3}
If $J \subset \tSS$ and $w \in {}^J \tW$ is a $\s$-Coxeter element in the finite Weyl group $W_{\supp_\s(w)}$, then
\[
X_{J, w}(\t) \cong \sqcup_{i \in \JJ_\t/(\JJ_\t \cap P_{\supp_\s(w) \cup I(J, w, \s)}}) i Y_{I(J, w, \s)}(w),
\]
where $Y_{I(J, w, \s)}(w) = \{ gP_{I(J, w, \s)} \in P_{\supp_\s(w)\cup I(J, w, \s)}/P_{I(J, w, \s)};\ g\i\t\s(g)\in P_{I(J, w, \s)}wP_{\s(I(J, w, \s))}  \}$.

Furthermore, the projection $\pi_J$ induces an isomorphism from the classical Deligne-Lusztig variety $\{g I \in P_{\supp_\s(w)}/I; g \i \t \s(g) \in I w I\}$ in the flag variety $P_{\supp_\s(w)}/I$ to
\begin{align*}
& \{ gP_{\supp_\s(w)\cap J} \in P_{\supp_\s(w)}/P_{\supp_\s(w)\cap J}; g\i\t\s(g)\in P_{\supp_\s(w)\cap J}\,\cdot_\s IwI
\} \\ & \cong   Y_{I(J, w, \s)}(w),
\end{align*}
and $Y_{I(J, w, \s)}(w)$ has dimension $\ell(w)$.
\end{cor}

\begin{proof}
The first part of the corollary follows from Theorem \ref{fineADLV} and Proposition \ref{decomp-flag2}.

Furthermore, it is easy to see that $I(J, w, \s)$ consists of $j \in J$ such that $s_j \text{ commutes with } s_i \text{ for all } i \in \supp_\s(w)$. In particular, $I(J, w, \s)$ and $\supp_\s(w)$ are disconnected in the affine Dynkin diagram. This implies that $\{g I \in P_{\supp_\s(w)}/I; g \i \t \s(g) \in I w I\}$ is isomorphic to $Y_{I(J, w, \s)}$. The proof of Theorem~\ref{fineADLV}~(1) implies that $\pi_J$ restricts to an isomorphism on $\{g I \in P_{\supp_\s(w)}/I; g \i \t \s(g) \in I w I\}$, and that the image of the latter variety can be identified with $Y_{I(J, w, \s)}(w)$.
\end{proof}

We follow the approach of \cite{Lu-par} and \cite{He-fineDL} for classical fine Deligne-Lusztig varieties in the partial flag variety. 

\subsection{The parahoric subgroup $P^Q$} For any $g \in G(L)$ and $H \subset G(L)$, we simply write ${}^g H$ for $g H g \i$. Let $G(L)'$ be the subgroup generated by all parahoric subgroups of $G(L)$. For any $J \subset \tSS$, let $\cp_J=\{{}^g P_J; g \in G(L)'\} \cong G(L)'/P_J$ be the set of parahoric subgroups conjugate to $P_J$ by an element of $G(L)'$. For any $J, K \subset \tSS$, $P \in \cp_J$ and $Q \in \cp_K$, we write $\pos(P, Q)=w$ if $w \in {}^J W_a^K$ and there exists $g \in G(L)'$ such that $P=g P_J g \i$ and $Q=g \dot{w} P_K \dot{w} \i g \i$, where $\dot{w}$ is a representative of $w$ in $G(L)'$. 

For any parahoric subgroups $P$ and $Q$, we set $P^Q=(P \cap Q) U_P$, where $U_P$ is the pro-unipotent radical of $P$. By \cite{Lu}, \S 1.1 (see also \cite{He-fineDL}, Lemma 2.3), one shows that $P^Q$ is again a parahoric subgroup. For any $g \in G(L)$, $$({}^g P)^{({}^g Q)}={}^g (P^Q).$$ %Similar to the proof of \cite[Lemma 2.3]{He-fineDL}, 

%\begin{lemma}\label{pq}
%Let $J, K \subset \tSS$ and $w \in {}^J W_a$. Set $w_1=\min(w W_K)$ and $J_1=J \cap w_1 K w_1 \i$. Then for any $g \in I w I$, we have \[P_J^{({}^g P_K)}=P_{J_1}.\]
%\end{lemma}

\subsection{B\'edard's description of ${}^J W_a$} For $J \subset \tSS$, let $\ct(J, \t \s)$ be the set of sequences $(J_n, w_n)_{n \ge 0}$ such that

(a) $J_0=J$,

(b) $J_n=J_{n-1} \cap w_{n-1}  \t \s(J_{n-1}) w_{n-1} \i$ for $n \ge 1$,

(c) $w_n \in {}^{J_n} W_a\,^{ \t \s(J_n)}$ for $n \ge 0$,

(d) $w_n \in W_{J_n} w_{n-1} W_{ \t \s(J_{n-1})}$ for $n \ge 1$.

Then for any sequence $(J_n, w_n)_{n \ge 0} \in \ct(J,  \t \s)$, we have that $w_n=w_{n+1}=\cdots$ and $J_n=J_{n+1}=\cdots$ for $n \gg 0$. By \cite{Be}, the assignment $(J_n, w_n)_{n \ge 0} \mapsto w_m$ for $m \gg 0$ defines a bijection $\ct(J,  \t \s) \to {}^J W_a$.

\subsection{Lusztig's partition of $\cp_J$}\label{8} Following \cite{Lu-par}, section 4,  we give a partition on $\cp_J$.

Let $b \in \t G(L)'$. To each $P \in \cp_J$, we associate a sequence $(P^n, J_n, w_n)_{n \ge 0}$ as follows
\begin{gather*} P^0=P, \quad P^n=(P^{n-1})^{(b \s(P^{n-1}) b \i)} \text{ for } n \ge 1,\\
J_n \subset I \text{ with } P^n \in \cp_{J_n}, \quad w_n=\pos(P^n, b \s(P^n) b \i) \qquad \text{ for } n \ge 0. \end{gather*}

By \cite[\S 1.4]{Lu}, $(J_n, w_n)_{n \ge 0} \in \ct(J, \t \s)$. For $w \in {}^J W_a$, let $$\cp_{J, w \t}(b)=\{P \in \cp_J; w_m=w \text{ for } m \gg 0\}.$$ Then $\cp_J=\sqcup_{w \in {}^J W_a} \cp_{J, w\t}(b)$.

\begin{prop}
Let $w \in {}^J W_a$. Then $$\cp_{J, w \t}(b)=\{{}^g P_J; g \in G(L)', g \i b \s(g) \in P_J \cdot_\s I w \t I\}.$$ 
\end{prop}

\begin{proof}
Notice that $G(L)' \cdot_\s b \subset \t G(L)'=\sqcup_{w \in {}^J W_a} P_J \cdot_\s I w \t I$. Then any $P \in \cp_J$ is of the form $P={}^g P_J$ for some $g \in G(L)$ with $g \i b \s(g) \in I w \t I$ for a unique $w \in {}^J W_a$. Let 
$(P^n, J_n, w_n)_{n \ge 0}$ be the sequence associated to ${}^g P_J$. Similar to \cite[Lemma 2.4]{He-fineDL}, $w_m=w$ for $m \gg 0$. Hence $P \in \cp_{J, w \t}(b)$. 
\end{proof}

\subsection{Part (1) of Theorem \ref{fineADLV}}\label{sec:4.5}
Let $p_J: \cp_\emptyset \to \cp_J$ be the projection map. Similarly to \cite[\S 4.2 (c) \& (d)]{Lu-par}, for any $n \ge 0$, the map $P \mapsto P^n$ gives an isomorphism $\vartheta_n: \cp_{J, w \t}(b) \to \cp_{J_n, w \t}(b)$ and the inverse map is $p_J$. In particular, $p_{J_n}=\vartheta_n \circ p_J$. 

By \cite[Lemma 1.4]{He-gstable}, $I(J, w \t, \s)=J_m$ for $m \gg 0$. By Lang's theorem, $P_{I(J, w \t, \s)} \cdot_\s I w \t I=P_{I(J, w \t, \s)} w \t P_{\s(I(J, w \t, \s))}$. Therefore $\cp_{J, w \t}(b)$ is isomorphic to \[\cp_{I(J, w \t, \s), w \t}(b)=\{{}^g P_{I(J, w \t, \s)} \in \cp_{I(J, w \t, \s)}; g \i b \s(g) \in P_{I(J, w \t, \s)} w \t P_{\s(I(J, w \t, \s))}\}.\] 

We fix $\t' \in \Omega$. Then 
\begin{align*} 
& X_{J, w \t}(b) \cap \t' G(L)'/P_J=\{g \t' P_J; g \in G(L)', (g \t') \i b \s(g \t') \in P_J \cdot_\s I w \t I\} \\ &=\{g P_{\t'(J)} \t'; g \in G(L)', g \i b \s(g) \in P_{\t'(J)} \cdot_\s I \t' w \t \s(\t') \i I\} \\ &\cong \{g P_{K} \t'; g \in G(L)', g \i b \s(g) \in P_{K} \t' w \t \s(\t') \i P_{\s(K)}\}
\\ &=\{g \t' P_{(\t') \i(K)}; g \in G(L)', g \i b \s(g) \in P_{(\t') \i(K)} w \t P_{\s(\t') \i(K)}\}  \\ &=X_{I(J, w \t, \s), w \t}(b) \cap \t' G(L)'/P_{I(J, w \t, \s)},
\end{align*}
here $K=I(\t'(J), \t' w \t \s(\t') \i, \s)=\t'(I(J, w \t, \s))$. Part (1) of the Theorem \ref{fineADLV} follows by combining all such $\t'$'s together. 

\subsection{Part (2) of Theorem \ref{fineADLV}}\label{sec:4.6}
In this subsection, we assume that $F=\FF_q((\e))$. Suppose that $w_m=w$.  Let $\bar P_{I(J, w \t, \s)}$ be the reductive quotient of $P_{I(J, w \t, \s)}$ and $\bar I$ the image of $I$ in $\bar P_{I(J, w \t, \s)}$. The fiber of the map $p_{I(J, w \t, \s)}: \cp_{\emptyset, w \t}(b) \to \cp_{I(J, w \t, \s), w \t}(b)$ is isomorphic to $\{p I \in P_{I(J, w \t, \s)}/I; p \i w \t \s(p) \in I w \t I\}=\{p I \in \bar P_{I(J, w \t, \s)}/ \bar I; (\text{Ad}(w \t) \circ \s)(p I)=p I\}$. Here $\text{Ad}(w \t) \circ \s$ is a twisted Frobenius morphism on $P_{I(J, w \t, \s)}/I \cong \bar P_{I(J, w \t, \s)}/\bar I$. In particular, the fiber of $\pi_{I(J, w \t, \s)}$ is $0$-dimensional. Since $\th_m$ is an isomorphism and $p_{I(J, w \t, \s)}=\th_m \circ p_J$, the fiber of $p_J$ is also $0$-dimensional. Hence $\dim X_{J, w \t}(b)=\dim X_{w \t}(b)$. 

\section{Newton strata and Ekedahl-Oort strata}

\subsection{Ekedahl-Oort strata}\label{setup} From now on, we assume that $G$ is absolutely quasi-simple and $J$ is a maximal proper subset of $\tSS$ and that $\s(J)=J$. Therefore $P_J(F)$ is a rational maximal parahoric subgroup\footnote{This assumption excludes some cases to which our method can be applied, for instance the Hilbert-Blumenthal case. In fact, one may consider maximal rational parahoric subgroups instead. However, in some cases, the only maximal rational parahoric subgroups are just rational Iwahori subgroups. Then the resulting stratification is the Kottwitz-Rapoport stratification. It is much harder to achieve a complete classification under this weaker assumption, so we do not consider it here.} of $G(F)$. 

We simply write $\EO^J(\mu)$ for $\Adm^J(\mu) \cap {}^J \tW$. The elements in $\EO^J(\mu)$ are called {\it EO elements}. Here EO stands for Ekedahl-Oort. For any $w \in \EO^J(\mu)$, we call $Z_{J, w}$ an {\it Ekedahl-Oort stratum}. 

Let $\EO^J_{\s, \cox}(\mu)$ be the subset of $\EO^J(\mu)$ consists of elements $w$, where $\supp_\s(w)$ is a proper subset of $\tSS$ and $w$ is a $\s$-Coxeter element of $W_{\supp_\s(w)}$. Since $G$ is absolutely quasi-simple, the affine Dynkin diagram of $\tW$ is connected. Hence if $w \in \EO^J_{\s, \cox}(\mu)$, then $W_{\supp_\s(w)}$ is a finite Weyl group. 

It follows from Viehmann's paper \cite{Vi}, Theorem 1.1, that this notion coincides with the usual notion of Ekedahl-Oort strata, if $G$ is unramified.

We are mainly interested in the case where the basic locus is the union
\begin{equation}\label{decomp_basic}
Z_{J, \co_0}=\sqcup_{w \in \EO^J_{\s, \cox}(\mu)} Z_{J, w}.
\end{equation}
In other words, the basic locus is a union of Ekedahl-Oort strata and each Ekedahl-Oort stratum is (``in a natural way'') the union of classical Deligne-Lusztig varieties attached to a Coxeter element. If~\eqref{decomp_basic} holds, then we say that $(G, \mu, J)$ is {\it of Coxeter type}. The first main result of this paper is the classification theorem. 

\begin{thm}\label{classification}
Let $(G, \mu, J)$ be as in $\S$\ref{setup}. The following 21 cases is the complete list (up to isomorphisms) of the triples $(G, \mu, J)$ of Coxeter type:\footnote{Here we use the names in \cite[Section 4]{Tits}.} 

\begin{tabular}{| c | c | c | }
\hline
  $(A_n, \o^\vee_1, \SS)$  & $(B_n, \o^\vee_1, \SS)$ & $(B_n, \o^\vee_1, \tSS-\{n\})$ \\
 \hline
  $(B$-$C_n, \o^\vee_1, \SS)$ & $(B$-$C_n, \o^\vee_1, \tSS-\{n\})$ & $(C$-$B_n, \o^\vee_1, \SS)$ \\
   \hline
 $(C$-$BC_n, \o^\vee_1, \SS)$ &    $(C$-$BC_n, \o^\vee_1, \tSS-\{n\})$ & $(D_n, \o^\vee_1, \SS)$  \\
   \hline
$({}^2 A'_n, \o^\vee_1, \SS)$  &   $({}^2 B_n, \o^\vee_1, \tSS-\{n\})$ & $({}^2 B$-$C_n, \o^\vee_1, \tSS-\{n\})$ \\
   \hline
$({}^2 D_n, \o^\vee_1, \SS)$ &     $(A_3, \o^\vee_2, \SS)$ & $({}^2 A'_3, \o^\vee_2, \SS)$ \\
   \hline
     $(C_2, \o^\vee_2, \SS)$  & $(C_2, \o^\vee_2, \tSS-\{1\})$ & $({}^2 C_2, \o^\vee_2, \tSS-\{1\})$ \\
   \hline
$({}^2 C$-$B_2, \o^\vee_1, \tSS-\{1\})$ && \\
   \hline
\end{tabular}
\end{thm}

The theorem is proved in section~\ref{6}.

\subsection{Ranked poset} Recall that a ranked poset is a partially ordered set (poset) equipped with a rank function $\rho$ such that whenever $y$ covers $x$, $\rho(y)=\rho(x)+1$. We say that the partial order of a poset is {\it almost linear} if the poset has a rank function $\rho$ such that for any $x, y$ in the poset, $x<y$ if and only if $\rho(x)<\rho(y)$. 

\begin{thm}\label{main}
Let $(G, \mu, J)$ be as in Theorem~\ref{classification}. Then 

\begin{enumerate}

\item Every Newton stratum is a union of Ekedahl-Oort strata. In other words, there is a map $\EO^J(\mu) \to B(G, \mu)_J$, $w \mapsto \co_w$ such that $Z_w \subset Z_{\co_w}$. 

%(2) The inverse image of $\co_0$ under the map $\EO^J(\mu) \to B(G, \mu)_J$ is $\EO^J_{\s, \cox}(\mu)$, where $\EO^J_{\s, \cox}(\mu)=\{w \in \EO^J(\mu); \supp_\s(w) \subsetneqq \tSS\}$. 

\item The partial order of $B(G, \mu)_J$ (inherited from $B(G)$) is almost linear. 

\item The partial order $\le_{J, \s}$ of $\EO^J_{\s, \cox}(\mu)$ coincides with the usual Bruhat order and is almost linear. Here the rank is the length function. 

\item For any $w \in \EO^J(\mu)-\EO^J_{\s, \cox}(\mu)$ and $b' \in \co_w$ the $\s$-centralizer $\JJ_{b'}$ acts transitively on $X_{J,w}(b')$. In particular, whenever there is a notion of dimension (in the RZ cases and in the function field case), this implies that $\dim X_{J,w}(b') = 0$.

\item For any $w \in \EO^J_{\s, \cox}(\mu)$, set $J(w, \s)=I(J, w, \s) \cup \supp_\s(w)$. Then $$X_{J, w}(\t)=\sqcup_{i \in \JJ_\t/(\JJ_\t \cap P_{J(w, \s)})} i Y(w),$$ where
\begin{align*}
Y(w)=\{ & g P_{I(J, w, \s)} \in P_{J(w, \s)}/P_{I(J, w, \s)};\\
& g \i \t \s(g) \in P_{I(J, w, \s)} w P_{\s(I(J, w, \s)}\}. 
\end{align*}

%\item For any $w \in \EO^J_{\s, \cox}(\mu)$, $w$ is a $\s$-Coxeter element in $W_{\supp_\s(w)}$. 

\end{enumerate}
\end{thm}

The theorem is proved in section~\ref{proof_main}.

\subsection{}\label{sec:5.3}
Many of the above cases have been investigated in the context of Shimura varieties:

The case $(G, \mu, J)=(C_2, \omega_2^\vee, \SS)$ has been studied by Katsura and Oort \cite{katsura-oort} and Kaiser \cite{Kaiser}; see also the paper \cite{kudla-rapoport} by Kudla and Rapoport (where the results are applied to computing intersection numbers of arithmetic cycles).

The case $(G, \mu, J)=(A_n, \o^\vee_1, \SS)$ is the $U(1, n)$, $p$ split case in which the basic locus is $0$-dimensional. This is the situation considered by Harris and Taylor in \cite{HT}.

The case $(G, \mu, J)=({}^2 A'_n, \o^\vee_1, \SS)$ is the $U(1, n)$, $p$ inert case and is studied by Vollaard and Wedhorn in \cite{Vollaard-Wedhorn}.

The cases $(G, \mu, J)=(B$-$C_n, \o^\vee_1, \tSS-\{n\}), ({}^2 B$-$C_n, \o^\vee_1, \tSS-\{n\})$ and $(C$-$BC_n, \o^\vee_1, \tSS-\{n\})$ are the $U(1, *)$, $p$ ramified cases and are studied by Rapoport, Terstiege and Wilson in \cite{Rapoport-Terstiege-Wilson}. See Section~\ref{sec:examples} where we discuss these cases in more detail.

The case $({}^2 A'_3, \o^\vee_2, \SS)$ is the $U(2, 2)$, $p$ inert case and is studied by Howard and Pappas in \cite{howard-pappas}. They transfer the problem to questions about an orthogonal group.
The case $(G, \mu, J)=(A_3, \o^\vee_2, \SS)$ is the $U(2, 2)$, $p$ split case which has not been written down in detail but which should not be hard to deal with.

The cases $(G, \mu, J)=(B$-$C_n, \o^\vee_1, \SS)$ and $(C$-$BC_n, \o^\vee_1, \SS)$ are the ``exotic good reduction'' cases for ramified $U(1, *)$ and it was conjectured in \cite{Rapoport-Terstiege-Wilson} that the description of basic locus is similar to the cases studied in loc.cit. %Also $U(2, 2)$ ramified? 

The remaining cases (including both PEL and non-PEL types) seem completely new. 

% FIXME further cases to which the method applies in principle ...

\section{The study of Ekedahl-Oort elements}\label{6}

\subsection{General strategy} It is showed in \cite[\S 6]{He-GeometryOfADLV} that the nonemptiness pattern of affine Deligne-Lusztig varieties only depends on the data $(\tW, w, \t, \s)$ and is independent of the choice of $G$ (i.e., independent of the orientation of the local Dynkin diagrams). Hence whether or not $(G, \mu, J)$ is of Coxeter type depends only on $(\tW, \l, J, \s)$. Here $\l$ is the image of $\mu$ in $X_*(T)_\G$. 

Let us consider the following two conditions:

Coxeter-Straight Condition \CSC: $\EO^J(\mu)-\EO^J_{\s, \cox}(\mu)$ consists of $\s$-straight elements. 

Coxeter Condition \CC: If $w \in \EO^J(\mu)$ with $\bar \nu_w$ central, then $w$ is a $\s$-Coxeter element in $W_{\supp_\s(w)}$. 

By Theorem \ref{straightADLV}, \CSC implies that $(G, \mu, J)$ is of Coxeter type . If $(G, \mu, J)$ is of Coxeter type, then it is obvious that \CC holds. We will show via a case-by-case analysis that \CC implies \CSC. 

\subsection{The quadruple $(\tW, \l, J, \s)$}\label{7.2} We use the same labeling of Coxeter graph as in \cite{Bour}. If $\o^\vee_i$ is minuscule, we denote the corresponding element in $\Omega$ by $\t_i$. Let $\s_0$ be the unique nontrivial diagram automorphism for the finite Dynkin diagram if $W_0$ is of type $A_n, D_n$ (with $n \ge 5$) or $E_6$. For type $D_4$, we also denote by $\s_0$ the diagram automorphism which interchanges $\a_3$ and $\a_4$. 

The pairs $(\tW, \l)$ coming from $(G, \mu)$ with $G$ absolutely quasi-simple and $\mu$ minuscule in $G$ are as follows: $(\tilde A_n, \o^\vee_i)_{1 \le i \le n}$, $(\tilde B_n, \o^\vee_i)_{1 \le i \le n}$, $(\tilde C_n, \o^\vee_i)_{1 \le i \le n}$, $(\tilde C_n, 2 \o^\vee_n)$, $(\tilde D_n, \o^\vee_i)_{i=1, n-1, n}$, $(\tilde E_6, \o^\vee_i)_{i=1, 6}$, $(\tilde E_7, \o^\vee_1)$, $(\tilde F_4, \o^\vee_1)$, $(\tilde G_2, \o^\vee_2)$. 

It was pointed to us by X. Zhu that to compute these tuples one only needs to understand the restriction of $\mu$ to a maximal split torus; it is not required to fully ``compute'' the reduced affine root system the group $G$ gives rise to. % FIXME add further details?

We will show that the basic locus is a union of Ekedahl-Oort strata if and only if the quadruple $(\tW, \l, J, \s)$ is one of the following (up to automorphisms of $\tW)$: 

\begin{tabular}{| c | c | c | }
\hline
$(\tilde A_n, \o^\vee_1, \SS, \id)$ & $(\tilde A_n, \o^\vee_1, \SS, \s_0)$ & $(\tilde B_n, \o^\vee_1, \SS, \id)$ \\
\hline
$(\tilde B_n, \o^\vee_1, \tSS-\{n\}, \id)$ &$(\tilde B_n, \o^\vee_1, \tSS-\{n\}, \t_1)$ & $(\tilde C_n, \o^\vee_1, \SS, \id)$ \\
\hline
 $(\tilde D_n, \o^\vee_1, \SS, \id)$ & $(\tilde D_n, \o^\vee_1, \SS, \s_0)$ & $(\tilde A_3, \o^\vee_2, \SS, \id)$ \\
\hline
 $(\tilde A_3, \o^\vee_2, \SS, \s_0)$  & $(\tilde C_2, \o^\vee_2, \SS, \id)$ & $(\tilde C_2, \o^\vee_2, \tSS-\{1\}, \id)$ \\
\hline
$(\tilde C_2, \o^\vee_2, \tSS-\{1\}, \t_2)$  && \\
\hline
\end{tabular}

Note that this list is shorter than the one in Theorem~\ref{main} because some automorphisms of $\tW$ will not lift to automorphisms of the affine root system.

If the quadruple $(\tW, \l, J, \s)$ is not in the list above, we give an element $w$ such that the Coxeter condition \CC fails for $w$. If the quadruple $(\tW, \l, J, \s)$ is in the list above, we compute $\EO^J(\mu)$, $\EO^J_{\s, \cox}(\mu)$ and the Newton points of the elements in $\EO^J(\mu)-\EO^J_{\s, \cox}(\mu)$. 

\

We follow \cite[1.5]{He-Hecke} for the reduced expression of $t^\mu$. For $1 \le a, b \le n$, set $$s_{[a, b]}=\begin{cases} s_a s_{a-1} \cdots s_b, & \text{ if } a \ge b, \\ 1, & \text{ otherwise}. \end{cases}$$

\subsection{Type $\tilde A_{n-1}$} For simplicity, we consider the extended affine Weyl group of $GL_n$ and $\o^\vee_i=(1^{(i)}, 0^{(n-i)})$ instead. We may assume that $J=\SS$ and $\mu=\o^\vee_i$ with $1 \le i \le \frac{n}{2}$ (after applying some automorphism of $\tW$).

Case 1: $\s=\id$. 

If $2 \le i<\frac{n}{2}$, then \CC fails for $s_0 s_{[n-1, n-\gcd(n, i)]} \t$. 

If $i=\frac{n}{2} \ge 3$, then \CC fails for $s_0 s_1 s_{n-1} s_0 \t$. 

If $i=\frac{n}{2}=2$, then $\EO^J_{\s, \cox}(\mu)=\{\t, s_0 \t\}$ and $\EO^J(\mu)-\EO^J_{\s, \cox}(\mu)=\{s_0 s_1 \t, s_0 s_3 \t, s_0 s_1 s_3 \t, s_0 s_1 s_3 s_0 \t\}$ consists of $\s$-straight elements. Moreover, $\nu_{s_0 s_1 \t}=(\frac{2}{3}^{(3)}, 0^{(1)})$, $\nu_{s_0 s_3 \t}=(1^{(1)}, \frac{1}{3}^{(3)})$, $\nu_{s_0 s_1 s_3 \t}=(1, \frac{1}{2}^{(2)}, 0)$ and $\nu_{s_0 s_1 s_3 s_0 \t}=(1^{(2)}, 0^{(2)})$. 

If $i=1$, then $\EO^J_{\s, \cox}(\mu)=\{\t\}$ and $\EO^J(\mu)-\EO^J_{\s, \cox}(\mu)=\{s_0 s_{[n-1, i]} \t; 2 \le i \le n\}$ consists of $\s$-straight elements. For $2 \le i \le n$, $\nu_{s_0 s_{[n-1, i]} \t}=(\frac{1}{i-1}^{(i-1)}, 0^{(n-i+1)})$. 

Case 2: $\s=\s_0$. 

If $i \ge 3$ or $n>4$, then \CC fails for $s_0 s_1 s_{n-1} s_0 \t$.

If $i=2, n=4$, then $\EO^J_{\s, \cox}(\mu)=\{\t, s_0 \t, s_0 s_1 \t, s_0 s_3 \t\}$ and $\EO^J(\mu)-\EO^J_{\s, \cox}(\mu)=\{s_0 s_1 s_3 \t, s_0 s_1 s_3 s_0 \t\}$ consists of $\s$-straight elements. Moreover, $\nu_{s_0 s_1 s_3 \t}=(1, \frac{1}{2}^{(2)}, 0)$ and $\nu_{s_0 s_1 s_3 s_0 \t}=(1^{(2)}, 0^{(2)})$ are all distinct. 

If $i=1$, then $\EO^J_{\s, \cox}(\mu)=\{\t, s_0 \t, \cdots, s_0 s_{[n-1, \lceil \frac{n+3}{2} \rceil]} \t\}$ and $\EO^J(\mu)-\EO^J_{\s, \cox}(\mu)=\{s_0 s_{[n-1, i]} \t; 2 \le i \le \lceil \frac{n+1}{2} \rceil\}$ consists of $\s$-straight elements. For $2 \le i \le \lceil \frac{n+1}{2} \rceil$, $\nu_{s_0 s_{[n-1, i]} \t}=(\frac{i}{2(i-1)}^{(i-1)}, \frac{1}{2}^{(n-2 i+2)}, \frac{i-2}{2(i-1)}^{(i-1)})$. 

\subsection{Type $\tilde B_n$} Let $J=\tSS-\{i\}$. If $i \notin \{0, 1, n\}$, then \CC fails for $s_{[n, i]} \i s_{[n-1, i]} \t$. We may assume that $J=\SS$ or $\tSS-\{n\}$  (after applying some automorphism of $\tW$). 

Case 1: $J=\SS$, $\s=\id$. 

If $\t=\id$, then $\mu \ge \o^\vee_2$ and \CC fails for $s_0 s_{[n, 2]} \i s_{[n-1, 2]} s_0$.

If $\t=\t_1$ and $\mu>\o^\vee_1$, then \CC fails for $s_0 s_2 s_1 \t$.

If $\mu=\o^\vee_1$, then $\EO^J_{\s, \cox}(\mu)=\{\t, s_0 \t, s_0 s_2 \t, \cdots, s_0 s_{[n-1, 2]} \i \t\}$ and $\EO^J(\mu)-\EO^J_{\s, \cox}(\mu)=\{s_0 s_{[n, 2]} \i \t s_{[n-1, i]}; 1 \le i \le n\}$ consists of $\s$-straight elements. For $1 \le i \le n$, $\nu_{s_0 s_{[n, 2]} \i \t s_{[n-1, i]}}=(\frac{1}{i}^{(i)}, 0^{(n-i)})$. 

Case 2: $J=\tSS-\{n\}$, $\s=\id$ or $\t_1$. 

If $\mu>\o^\vee_1$, then \CC fails for $s_n s_{n-1} s_n \t$.

If $\mu=\o^\vee_1$ and $\s=\id$, then $\EO^J_{\s, \cox}(\mu)=\{\t, s_n \t, s_n s_{n-1} \t, \cdots, s_{[n, 2]} \t\}$ and $\EO^J(\mu)-\EO^J_{\s, \cox}(\mu)=\{s_{[n, 2]} s_1 \t, s_{[n, 2]} s_0 \t\} \cup \{s_{[n, 0]} s_{[i, 2]} \i \t; 2 \le i \le n-1\}$ consists of $\s$-straight elements. Moreover, $\nu_{s_{[n, 2]} s_1 \t}=\nu_{s_{[n, 2]} s_0 \t}=(\frac{1}{n}^{(n)})$ and for $2 \le i \le n-1$, $\nu_{s_{[n, 0]} s_{[i, 2]} \i \t}=(\frac{1}{n-i}^{(n-i)}, 0^{(i)})$. 

If $\mu=\o^\vee_1$ and $\s=\t_1$, then $\EO^J_{\s, \cox}(\mu)=\{\t, s_n \t, \cdots, s_{[n, 2]} \t\} \cup \{s_{[n, 2]} s_1 \t, s_{[n, 2]} s_0 \t\}$ and $\EO^J(\mu)-\EO^J_{\s, \cox}(\mu)=\{s_{[n, 0]} s_{[i, 2]} \i \t; 1 \le i \le n-1\}$ consists of $\s$-straight elements. For $1 \le i \le n-1$, $\nu_{s_{[n, 0]} s_{[i, 2]} \i \t}=(\frac{1}{n-i}^{(n-i)}, 0^{(i)})$.

\subsection{Type $\tilde C_n$} Let $J=\tSS-\{i\}$ with $i \le \frac{n}{2}$ (after applying some automorphism of $\tW$). 

Case 1: $\mu \ge \o^\vee_1$. 

If $i \neq 0$, then \CC fails for $s_{[i, 0]} s_{[1, i]} \t$. Hence $J=\SS$ and $\s=\id$. 

If $\mu>\o^\vee_1$, then \CC fails for  $s_0 s_1 s_0$.

If $\mu=\o^\vee_1$, then $\EO^J_{\s, \cox}(\mu)=\{1, s_0, s_0 s_1, \cdots, s_{[n-1, 0]} \i\}$ and $\EO^J(\mu)-\EO^J_{\s, \cox}(\mu)=\{s_{[n, 0]} \i s_{[n-1, i]}; 1 \le i \le n\}$ consists of $\s$-straight elements. For $1 \le i \le n$, $\nu_{s_{[n, 0]} \i s_{[n-1, i]}}=(\frac{1}{i}^{(i)}, 0^{(n-i)})$. 

Case 2: $\mu=\o^\vee_n$. 

If $0<i<\frac{n}{2}$, then $\s=\id$ and \CC fails for $s_{[n-i, i]} \i \t$. 

If $i=\frac{n}{2}>1$, then \CC fails for $s_i s_{i+1} s_{i-1} s_i \t$. 

If $i=0$ and $n>2$, then \CC fails for $s_0 s_1 s_0 \t$. 

Therefore $n=2$. 

If $J=\SS$, then $\s=\id$ and $\EO^J_{\s, \cox}(\mu)=\{\t, s_0 \t\}$ and $\EO^J(\mu)-\EO^J_{\s, \cox}(\mu)=\{s_0 s_1 \t, s_0 s_1 s_0 \t\}$ consists of $\s$-straight elements. Moreover, $\nu_{s_0 s_1 \t}=(\frac{1}{2}, 0)$ and $\nu_{s_0 s_1 s_0 \t}=(\frac{1}{2}^{(2)})$.

If $J=\tSS-\{1\}$ and $\s=\id$, then $\EO^J_{\s, \cox}(\mu)=\{\t, s_1 \t\}$ and $\EO^J(\mu)-\EO^J_{\s, \cox}(\mu)=\{s_1 s_0 \t, s_1 s_2 s_0 \t\}$ consists of $\s$-straight elements. Moreover, $\nu_{s_1 s_2 \t}=\nu_{s_1 s_0 \t}=(\frac{1}{2}, 0)$ and $\nu_{s_1 s_2 s_0 \t}=(\frac{1}{2}^{(2)})$. 

If $J=\tSS-\{1\}$ and $\s=\t_2$, then $\EO^J_{\s, \cox}(\mu)=\{\t, s_1 \t, s_1 s_2 \t, s_1 s_0 \t\}$ and $\EO^J(\mu)-\EO^J_{\s, \cox}(\mu)=\{s_1 s_2 s_0 \t\}$ consists of $\s$-straight elements. Moreover, $\nu_{s_1 s_2 s_0 \t}=(\frac{1}{2}^{(2)})$.

\subsection{Type $\tilde D_n$} Let $J=\tSS-\{i\}$. After applying some automorphism of $\tW$, we may assume that $i=0$ or $2 \le i \le \frac{n}{2}$ and $\mu=\o^\vee_1$ or $\o^\vee_n$. 

Case 1: $\mu=\o^\vee_1$. 

If $i \neq 0$, then \CC fails for $s_{[n, i]} \i s_{[n-2, i]} \t$. Hence $J=\SS$.

If $\s=\id$, then $\EO^J_{\s, \cox}(\mu)=\{\t, s_0 \t, s_0 s_2 \t, \cdots, s_0 s_{[n-2, 2]} \i \t\}$ and $\EO^J(\mu)-\EO^J_{\s, \cox}(\mu)=\{s_0 s_{[n-2, 2]} \i s_{n-1} \t, s_0 s_{[n-2, 2]} \i s_{n} \t\} \cup \{s_0 s_{[n, 2]} \i \t s_{[n-2, i]}; 1 \le i \le n-1\}$ consists of $\s$-straight elements. Moreover, $\nu_{s_0 s_{[n-2, 2]} \i s_{n-1} \t}=\nu_{s_0 s_{[n-2, 2]} \i s_{n} \t}=(\frac{1}{n}^{(n)})$ and for $1 \le i \le n-1$, $\nu_{s_0 s_{[n, 2]} \i \t s_{[n-2, i]}}=(\frac{1}{i}^{(i)}, 0^{(n-i)})$.

If $\s=\s_0$, then $\EO^J_{\s, \cox}(\mu)=\{\t, s_0 \t, \cdots, s_0 s_{[n-1, 2]} \i \t,  s_0 s_{[n-2, 2]} \i s_n \t\}$ and $\EO^J(\mu)-\EO^J_{\s, \cox}(\mu)=\{s_0 s_{[n, 2]} \i \t s_{[n-2, i]}; 1\le i \le n-1\}$ consists of $\s$-straight elements. Moreover, for $1 \le i \le n-1$, $\nu_{s_0 s_{[n, 2]} \i \t s_{[n-2, i]}}=(\frac{1}{i}^{(i)}, 0^{(n-i)})$.

If $\s(1) \neq 1$, then $n=4$ and \CC fails for $s_0 s_2 s_3$ or $s_0 s_2 s_4$.

Case 2: $\mu=\o^\vee_n$. 

%Set $$\e=\begin{cases} 1, & \text{ if } 2 \nmid n; \\ 0, & \text{ if } 2 \mid n.\end{cases}$$

If $2 \le i<\frac{n}{2}$, then \CC fails for $s_{[n-i, i]} \i \t$. 

If $i=\frac{n}{2}$, then \CC fails for $s_i s_{i+1} s_{i-1} s_i \t$. 

Now we consider the case where $J=\SS$ and $n>4$. One may check that \CC fails for $s_0 s_2 s_1 \t$ or $s_0 s_2 s_1 s_0 \t$. 

\subsection{Type $\tilde E_6$} After applying some automorphism of $\tW$, we may assume that $\mu=\o^\vee_1$ and $J=\SS, \tSS-\{2\}$ or $\tSS-\{4\}$. 

If $J=\SS$, then \CC fails for $s_0 s_2 s_4 s_3 s_1 \t$.

If $J=\tSS-\{2\}$, then \CC fails for $s_2 s_4 s_5 \t$ or $s_2 s_4 s_3 \t$.

If $J=\tSS-\{4\}$, then \CC fails for $s_4 s_3 s_5 s_4 \t$. 

\subsection{Type $\tilde E_7$} Here $\mu=\o^\vee_1$. After applying some automorphism of $\tW$, $J=\tSS-\{i\}$ for $0 \le i \le 4$. 

If $J=\SS$, then \CC fails for $s_0 s_1 s_3 s_4 s_2 s_5 s_4 s_3 s_1 s_0 \t$.

If $J=\tSS-\{1\}$, then \CC fails for $s_1 s_3 s_4 s_5 s_6 \t$. 

If $J=\tSS-\{2\}$, then \CC fails for $s_2 s_4 s_3 s_5 s_4 s_2 \t$. 

If $J=\tSS-\{3\}$, then \CC fails for $s_3 s_4 s_5 \t$.

If $J=\tSS-\{4\}$, then \CC fails for $s_4 s_2 s_5 s_4 \t$. 

\subsection{Type $\tilde F_4$} Here $\mu=\o^\vee_1$. If $J=\SS$, then \CC fails for $s_0 s_1 s_2 s_3 s_2 s_1$. If $J=\tSS-\{1\}$, then \CC fails for $s_1 s_2 s_3 s_2$. If $J=\tSS-\{2\}$, then \CC fails for $s_2 s_3 s_2$. If $J=\tSS-\{3\}$, then \CC fails for $s_3 s_2 s_3$. If $J=\tSS-\{4\}$, then \CC fails for $s_4 s_3 s_2 s_3$. 

\subsection{Type $\tilde G_2$} Here $\mu=\o^\vee_2$. If $J=\SS$, then \CC fails for $s_0 s_2 s_1 s_2$. If $J=\tSS-\{1\}$, then \CC fails for $s_1 s_2 s_1 s_0$. If $J=\tSS-\{2\}$, then \CC fails for $s_2 s_1 s_2 s_0$.

\subsection{Proof of Theorem \ref{main}}\label{proof_main} Part (1) and (2) follow from the Coxeter-Straight condition. Part (3) follows from the explicit computation in each case. Part (4) follows from Theorem \ref{straightADLV}. Part (5) follows from Theorem \ref{fineADLV} and Proposition \ref{decomp-flag2}. Part (6) follows from the explicit description of $\EO^J_{\s, \cox}(\mu)$. Part (7) follows from $\S$\ref{4.3} (1), (2) and the fact that for all the cases we are considering, $\ell(x) < \ell(y)$ implies that $x < y$. 

\section{Closure relations}\label{7}

\subsection{The $\t \s$-orbits on $\tSS$}\label{sec:7.1}
In this section we assume that $F=\FF_q((\e))$ and that $(G, \mu, J)$ is as in Theorem~\ref{classification}. 
Let $\cn$ be the fiber over $\t$ for the map $Z_{\co_0} \to \co_0$; with the notation of \cite{rapoport:guide}, this means $\cn=X(\mu,\t)_{P_J}$. We usually refer to $\cn$ as the \emph{basic locus}. In this section, we study the stratification of $\cn$ by classical Deligne-Lusztig varieties in more detail. 

This description proceeds as follows: first we describe the set $\ci$ of EO strata which occur in the basic locus $\cn$ (in terms of the Dynkin diagram). Second, we will describe the set of strata within each EO stratum (in terms of the Bruhat-Tits building of $\JJ_\t$). Finally we will discuss the closure relations between strata.

Identify $\tSS$ with the set of vertices of the affine Dynkin diagram, and for any vertex $v$, denote by $d(v)$ the distance between $v$ and the unique vertex not contained in $J$.

% Let $\orb$ be the set of $\t \s$-orbits on $\tSS$.  We write $\Sigma_0$ for the $\t \s$-orbit that contains $\tSS-J$.  For any $\Sigma \in \orb$, let $d(\Sigma)$ be the distance between $\Sigma$ and $\Sigma_0$ in the affine Dynkin diagram (i.e., the distance between the vertex in $\Sigma_0$ and any of the vertices in $\Sigma$).

Let
\[
\ci = \{ \Sigma\subset \tSS;\ \emptyset\ne\Sigma\text{ is $\t\s$-stable and } \forall v,v'\in\Sigma\colon d(v)=d(v')    \}
\]
Clearly every $\t\s$-orbit is an element of $\ci$. In some cases, there is also one further element which is the union of two orbits, see Section \ref{sec:examples}. We denote by $d(\Sigma)$ the value of $d(v)$ for any $v\in\Sigma$.

Given $\Sigma\in\ci$, we denote by $\Sigma^{\flat}$ the union of all the $\t \s$-orbits $\Sigma'$ with $d(\Sigma') \le d(\Sigma)$ and $\Sigma' \not\subseteq \Sigma$ and denote by $\Sigma^{\sharp}$ the union of all the $\t \s$-orbits $\Sigma'$ with $d(\Sigma')>d(\Sigma)$.

From the explicit computation in $\S$\ref{6}, we get the following properties:

\begin{prop}
Let $\Sigma\in \ci$.
\begin{enumerate}
\item
We have $\tSS = \Sigma \sqcup \Sigma^\flat \sqcup \Sigma^\sharp$.
\item
The subsets $\Sigma^\flat$ and $\Sigma^\sharp$ are disconnected in the affine Dynkin diagram.
\item
For any $\Sigma \in \ci$, there is exactly one element $w \in \EO^J_{\s, \cox}(\mu)$ such that $\supp_\s(w)=\Sigma^{\flat}$. We denote this element by $w_\Sigma$. We have $\ell(w_\Sigma)=d(\Sigma)$. 
\item
We have $I(J, w_\Sigma, \s)=\Sigma^\sharp$ and $J(w_\Sigma, \s)=\tSS-\Sigma$ (cf.~Prop.~\ref{min}, Theorem \ref{main}).
\end{enumerate}
\end{prop}

\begin{example}\label{example_sigmaflat}
As an example, consider the case $(\tilde{B}_n, \omega^\vee_1, \tSS-\{n\}, \t_1)$. In this case, the possible triples $(\Sigma, \Sigma^\flat, \Sigma^\sharp)$ are $(\{i \}, \{ m, m-1, \dots, i+1 \}, \{ i-1, \dots, 1,0\})$ for $i=m,\dots, 2$, and $(\{0 \}, \{ m, m-1, \dots, 1 \}, \emptyset)$, $(\{1 \}, \{ m, m-1, \dots, 2, 0 \}, \emptyset)$ and $(\{0,1 \}, \{ m, m-1, \dots, 2 \}, \emptyset)$. In particular, in the last case, $\Sigma$ has more than one element.
\end{example}

\subsection{A stratification of $\cn$} Set $\cn_{\Sigma}=\cn \cap Z_{J, w_\Sigma}=X_{J, w_\Sigma}(\t)$, the Ekedahl-Oort stratum attached to $\Sigma$. Then $\cn=\sqcup_{\Sigma \in \ci} \cn_{\Sigma}$.  By Theorem \ref{main} (5),
\begin{equation}
\label{strat}
\cn_{\Sigma}=\sqcup_{i \in \JJ_\t/(\JJ_\t \cap P_{\tSS-\Sigma})} i Y(w_\Sigma),
\end{equation}
where by Corollary \ref{4.1.3}
\begin{align*}
Y(w_\Sigma) &=\{g P_{\Sigma^\sharp} \in P_{\tSS-\Sigma}/P_{\Sigma^\sharp}; g \i \t \s(g) \in P_{\Sigma^\sharp} w P_{\s(\Sigma^\sharp)}\} \\ &\cong  \{g I \in P_{\Sigma^\flat}/I; g \i \t \s(g) \in I w_\Sigma I\}.
\end{align*}

Since $\supp_\s(w_\Sigma)=\Sigma^\flat$, $Y(w_\Sigma)$ is connected and $\cn_{\Sigma}=\sqcup_{i \in \JJ_\t/(\JJ_\t \cap P_{\tSS-\Sigma})} i Y(w_\Sigma)$ is the decomposition of $\cn_\Sigma$ into connected components. Each of them has dimension $\ell(w_\Sigma)$.

%In particular, 

%\begin{thm}
%$\dim \cn=\sharp(\tSS/\<\t \s\>)-1$. 
%\end{thm}

%\begin{rem}
%If $(G, \mu, J)$ is of the form $(C$-$B_n, \o^\vee_1, \tSS-\{1\})$ or $(C$-$BC_n, \o^\vee_1, \tSS-\{1\})$, then $\dim \cn>\sharp(\tSS/\<\t \s\>)-1$. 
%\end{rem}

Now we describe the closure of each stratum in $\cn$. 

\begin{thm}
Let $\Sigma \in \ci$. Then for any $i \in \JJ_\t/(\JJ_\t \cap P_{\tSS-\Sigma})$, $$\overline{i Y(w_\Sigma)}=\sqcup_{(\Sigma')^\flat \subset \Sigma^\flat} \sqcup_{j \in  \JJ_\t/(\JJ_\t \cap P_{\tSS-\Sigma'}); i \cap j \neq \emptyset} j Y(w_{\Sigma'}).$$
\end{thm}

The intersection $i\cap j$ is understood as the intersection inside $\JJ_\t$ of the cosets given by $i$ and $j$.

\begin{proof}
It suffices to consider the case where $i=1$. Since $\pi_J$ is proper, $\overline{Y(w_\Sigma)}=\pi_J \bigl(\{g I \in P_{\Sigma^\flat}/I; g \i \t \s(g) \in \overline{I w_\Sigma I}\} \bigr)=\cup_{w' \le w_\Sigma} \pi_J \bigl(\{g I \in P_{\Sigma^\flat}/I; g \i \t \s(g) \in I w' I\} \bigr)$. Any element $w' \le w_\Sigma$ is of the form $w_{\Sigma'} w''$ for some $\Sigma'$ with $(\Sigma')^\flat \subset \Sigma^\flat$ and $w'' \in W_{(\Sigma')^\sharp}$. Since $(\Sigma')^\sharp \subset J$, $P_J \cdot_\s I w' I=P_J \cdot_\s I w_{\Sigma'} I$ and $\pi_J \bigl(\{g I \in P_{\Sigma^\flat}/I; g \i \t \s(g) \in I w' I\} \bigr)=\pi_J \bigl(\{g I \in P_{\Sigma^\flat}/I; g \i \t \s(g) \in I w_{\Sigma'} I\} \bigr)$. 

Similarly to the proof of Proposition \ref{decomp-flag}, $\{g I \in P_{\Sigma^\flat}/I; g \i \t \s(g) \in I w_{\Sigma'} I\}=\sqcup_{j \in (\JJ_\t \cap P_{\Sigma^\flat})/(\JJ_\t \cap P_{(\Sigma')^\flat})} \{j g I \in P_{(\Sigma')^\flat}/I; g \i \t \s(g) \in I w_{\Sigma'} I\}$ and $\pi_J \bigl(\{g I \in P_{\Sigma^\flat}/I; g \i \t \s(g) \in I w_{\Sigma'} I\} \bigr)=\sqcup_{j \in (\JJ_\t \cap P_{\Sigma^\flat})/(\JJ_\t \cap P_{(\Sigma)^\flat} \cap P_{\tSS-\Sigma'})} j Y(w_{\Sigma'})$.
Note that $\tSS-\Sigma=\Sigma^\flat \sqcup \Sigma^\sharp$ and $\Sigma^\sharp \subset \tSS-\Sigma'$. Hence $(\JJ_\t \cap P_{\tSS-\Sigma})/(\JJ_\t \cap P_{\tSS-\Sigma-\Sigma'}) \cong (\JJ_\t \cap P_{\Sigma^\flat})/(\JJ_\t \cap P_{(\Sigma)^\flat} \cap P_{\tSS-\Sigma'})$. The theorem is proved. 
\end{proof}

\

Another way to describe the closure of strata is via the Bruhat-Tits building of the group $\JJ_\t$ over $F$. This reproduces precisely the descriptions in \cite{Vollaard-Wedhorn} and \cite{Rapoport-Terstiege-Wilson}. 

\begin{prop}\label{closure_Z}
Let $\Sigma, \Sigma' \in \ci$ and $j, j'\in \JJ_\t$. The following are equivalent:
\begin{enumerate}
\item
$j (\JJ_\t \cap P_{\tSS-\Sigma}) \cap j' (\JJ_\t \cap P_{\tSS-\Sigma'})\ne\emptyset$,
\item
$j (\JJ_\t \cap P_{\tSS-\Sigma}) j\i \cap j' (\JJ_\t \cap P_{\tSS-\Sigma'}) (j')\i$ contains a (rational) Iwahori subgroup of $\JJ_\t$,
\item
The faces in the rational building of $\JJ_\t$ corresponding to $j (\JJ_\t \cap P_{\tSS-\Sigma}) j\i$ and $j' (\JJ_\t \cap P_{\tSS-\Sigma'}) (j')\i$ are neighbors (i.e., there exists an alcove which contains both of them). 
\end{enumerate}
\end{prop}

\begin{proof}
We may and will assume throughout that $j'=1$.

The Iwahori subgroups of $j (\JJ_\t \cap P_{\tSS-\Sigma}) j\i$ are of the forms $j g (\JJ_\t \cap I) g \i j \i$ for some $g \in \JJ_\t \cap P_{\tSS-\Sigma}$ and the Iwahori subgroups of $\JJ_\t \cap P_{\tSS-\Sigma'}$ are of the forms $g' (\JJ_\t \cap I) (g') \i$ for some $g' \in \JJ_\t \cap P_{\tSS-\Sigma'}$. Hence (2) is equivalent to saying that for some $g \in \JJ_\t \cap P_{\Sigma^\flat}$ and $g' \in \JJ_\t \cap P_{\tSS-\Sigma'}$, $j g (\JJ_\t \cap I) g \i j \i=g' (\JJ_\t \cap I) (g') \i$, i.e., $(g') \i j g \in I$. The latter condition is equivalent to condition (1). 

By definition of the simplicial structure of the Bruhat-Tits building, (2) and (3) are equivalent .
\end{proof}

\subsection{Singularities of the closures of strata}

Consider $(G, \mu, J)$ as in Theorem~\ref{classification}.

\begin{prop}
Each fine Deligne-Lusztig variety $Y(w_\Sigma)$ occurring in the stratification~\eqref{strat} is smooth. Its closure $\tilde{Y}$ inside the flag variety $P_{\Sigma^\flat}/I$ is smooth. In particular, the projection $\cp_\emptyset \to \cp_J$ restricts to a resolution of singularities $\tilde{Y} \to \overline{Y(w_\Sigma)}$ of the closure of $Y(w_\Sigma)$ in $\cp_J$ (or equivalently in $\cn$).
\end{prop}

\begin{proof}
The smoothness of $\tilde{Y}$ is equivalent to the smoothness of the Schubert variety inside $P_{\Sigma^\flat}/I$ attached to $w_\Sigma$. Since the latter is a Coxeter element, this Schubert variety is isomorphic to its Bott-Samelson resolution, and hence is smooth.
\end{proof}

\begin{prop}
If the triple $(\tW, \l, J, \s)$ belongs to the following list

\medskip
\begin{tabular}{| c | c | c | }
\hline
$(\tilde A_n, \o^\vee_1, \SS, \id)$\quad (*) & $(\tilde A_n, \o^\vee_1, \SS, \s_0)$ \quad (*) & $(\tilde B_n, \o^\vee_1, \tSS-\{n\}, \id)$ \\
\hline
$(\tilde B_n, \o^\vee_1, \tSS-\{n\}, \t_1)$ & $(\tilde C_n, \o^\vee_1, \SS, \id)$ &  $(\tilde A_3, \o^\vee_2, \SS, \id)$ \quad(*) \\
\hline
$(\tilde A_3, \o^\vee_2, \SS, \s_0)$ \quad(*)  & $(\tilde C_2, \o^\vee_2, \SS, \id)$ \quad (*) & $(\tilde C_2, \o^\vee_2, \tSS-\{1\}, \id)$ (*) \\
\hline
$(\tilde C_2, \o^\vee_2, \tSS-\{1\}, \t_2)$  && \\
\hline
\end{tabular}

\medskip
then
\begin{enumerate}
\item
for all $\Sigma\in\ci$, the closure $\overline{Y(w_\Sigma)}$ of $Y(w_\Sigma)$ inside $\cn$ has at most isolated singularities, and
\item
the closure $\overline{Y(w_\Sigma)}$ is smooth if and only if $\t(J) \ne J$ or $\ell(w_\Sigma)\le 1$.
\end{enumerate}
The cases where all closures $\overline{Y(w_\Sigma)}$ are smooth are marked (*) in the table.
\end{prop}

\begin{proof}
We write $w:=w_\Sigma$ and assume that $\ell(w) > 0$. Denote by $\mathrm F$ the twisted Frobenius $g\mapsto \t \s(g)\t\i$; it acts on $\supp_\s(w)$ and hence on the flag variety $P_{\supp_\s(w)}/I$.
Let $Q:= P_{\supp_\s(w)\cap J}$ and $w':= w\t\i$, and denote by
\[
X_Q(w') := \{ gQ \in P_{\supp_\s(w)}/Q;\ g\i \mathrm F(g)\in Qw \mathrm F(Q)  \}
\]
the Deligne-Lusztig variety attached to $w'$ inside $P_{\supp_\s(w)}/Q$ with respect to $\mathrm F$. Recall that $w$ is $\s$-Coxeter which by definition means that $w'$ is a twisted Coxeter element for $\mathrm F$.

Using Cor.~\ref{4.1.3}, we identify $Y(w)$ with
\[
\{ gQ \in P_{\supp_\s(w)}/Q;\ g\i \mathrm F(g)\in Q\cdot_\s Iw'I  \}.
\]

Let $\overline{Y(w)}$, $\overline{X_Q(w')}$ denote the closures inside $P_{\supp_\s(w)}/Q$. Note that $\overline{Y(w)}$ is isomorphic to the closure of $Y(w)$ inside $\cn$. Clearly we have
\begin{equation}\label{YinX}
\overline{Y(w)} \subseteq \overline{X_Q(w')}.
\end{equation}

Furthermore $\overline{X_Q(w')}$ is irreducible by the well-known criterion for irreducibility of Deligne-Lusztig varieties and $\dim \overline{X_Q(w')} = \dim Qw'\mathrm F(Q)/\mathrm F(Q)$.

By Corollary~\ref{4.1.3}, we have $\dim Y(w) = \ell(w)$. Thus we see that the inclusion~\eqref{YinX} is an equality if and only if $\ell(w) = \dim Qw'\mathrm F(Q)/\mathrm F(Q)$. 

Let $w_0$ denote the longest element in $W_{\supp_\s(w)\cap J}$. Since $w\in {}^J\tW$, we have $\ell(w_0w') = \ell(w_0) + \ell(w') = \ell(w') + \dim Q/I$, and we conclude that the inclusion~\eqref{YinX} is an equality if and only if $w_0w'$ is the longest element inside $W_{\supp_\s(w)\cap J}w' W_{\mathrm F(\supp_\s(w)\cap J)}$. In the cases listed in the statement of the proposition, the Dynkin type of $\supp_\s(w)\cap J$ is type $A$, and it is easily checked that in those cases, $w_0w'$ indeed is the longest element inside this double coset.

Now suppose that $\t(J) \ne J$, or equivalently that $\mathrm F(Q)\ne Q$. Since $\tSS - J \subseteq \supp(w)$, we have $w'\in W_{\mathrm F(\supp_\s(w)\cap J)}$, so $X_Q(w') = X_Q(\id)$ is smooth and closed in $P_{\supp_\s(w)}/Q$.

On the other hand, if $\mathrm F(Q)=Q$, then the above implies that $\overline{Y(w)}=\overline{X_Q(w')} = X_Q(w') \sqcup X_Q(\id)$. Since $\dim X_Q(\id) = 0$, part (1) follows. Finally, since closures of Deligne-Lusztig varieties are always normal, $\overline{Y(w)}$ is smooth if $\ell(w)\le 1$. The remaining assertion in part (2) follows from Proposition 3.3 and Proposition 4.4 in \cite{brion-polo} by Brion and Polo (note that $\ell(w) < \dim P_{\supp_\s(w)}/Q$ whenever $\ell(w) > 1$, hence the boundary in the sense of loc.~cit.~is $X_Q(\id)$).
\end{proof}

%\bigskip
%Another point is the question in which cases all irreducible components are rational (i.e. birational to projective space). Cf. Rapoport's ICM talk and the paper by Hales (Hyperelliptic curves and harmonic analysis) cited there.

\subsection{Examples}\label{sec:examples}
As examples, we discuss the three cases treated by Rapoport, Terstiege and Wilson in \cite{Rapoport-Terstiege-Wilson} and an example of non-Coxeter type.

\subsubsection{$(C$-$BC_m, \o^\vee_1, \tSS-\{m\})$} This case arises from $GU(1, 2m)$, $p$ ramified. Note here that the vertices $0$ and $m$ are conjugate under the extended affine Weyl group, but not conjugate under the diagram automorphism for the relative local Dynkin diagram. In this case there exist $m+1$ different EO strata in the basic locus, of dimensions $0$, \dots, $m$. In the terminology of \cite{Rapoport-Terstiege-Wilson}, the EO stratum of dimension $i$ is the union of all strata attached to vertex lattices of type $2i$. The stratification is indexed by the vertices of the Bruhat-Tits building of $\JJ_\t$.

\subsubsection{$({}^2 B$-$C_m, \o^\vee_1, \tSS-\{m\})$} This case arises from $GU(1, 2m-1)$, $p$ ramified, the hermitian form $C$ is split. In this case there exist $m+2$ EO strata in the basic locus; one each of dimension $0$, \dots, $m-1$, and two of dimension $m$. For $i=0,\dots, m-1$, the EO stratum of dimension $i$ is the union of all strata attached to vertex lattices of type $2i$. On the other hand, the union of all strata attached to vertex lattices of type $2m$ is equal to the union of the two $m$-dimensional EO strata (the corresponding EO strata are related by the conjugation action of $GU(1, 2m-1)$).  In this case, the action of $\s\t$ on the affine Dynkin diagram is trivial. There are two orbits with the same distance to the vertex $m$ (namely the vertices $0$ and $1$). See Example~\ref{example_sigmaflat}. The union of these two orbits occurs as an element $\Sigma\in\ci$; it corresponds to the $m-1$-dimensional EO stratum. The index set for the Deligne-Lusztig varieties inside this EO stratum is the set of edges of type $\{0,1\}$ inside the Bruhat-Tits building of $\JJ_\t$. In the description in \cite{Rapoport-Terstiege-Wilson} this is reflected by Prop.~3.4.

\subsubsection{$(B$-$C_m, \o^\vee_1, \tSS-\{m\})$} This case arises from $GU(1, 2m-1)$, $p$ ramified, the hermitian form $C$ being nonsplit. In this case basic EO strata correspond bijectively to types of vertex lattices, and the index set of the stratification is the set of vertices of the Bruhat-Tits building of $\JJ_\t$.

\subsubsection{$(C$-$BC_2, \o^\vee_1, \tSS-\{1\})$} This case arises from $GU(1, 3)$, $p$ ramified. However, the level structure is different from the one considered in \cite{Rapoport-Terstiege-Wilson}. One can show that the basis locus is still the union of EO strata. There are 6 EO strata in the basic locus, which correspond to the elements $1$, $s_1$, $s_1 s_0$, $s_1 s_2$, $s_1 s_0 s_1$ and $s_1 s_2 s_1$ respectively. One is of dimension $0$, one is of dimension $1$, two of dimension $2$ and two of dimension $3$. The elements $s_1 s_0 s_1$ and $s_1 s_2 s_1$ are not Coxeter elements.

\end{document}